\documentclass{article}
\usepackage[utf8]{inputenc}
\usepackage{authblk}

\usepackage{amsfonts,amsmath,amssymb}
\usepackage{graphicx}
\usepackage{authblk}

\newtheorem{satz}{Theorem}

\newtheorem{theorem}[satz]{Theorem}
\newtheorem{lemma}[satz]{Lemma}
\newtheorem{definition}[satz]{Definition}
\newtheorem{corollary}[satz]{Corollary}

\newtheorem{claim}[satz]{Claim}

\newcommand{\qed}{{} \hfill \mbox{$\Box$}}

\def\({\big (}
\def\){\big )}

\def\_phi{\varphi}

\title{Sumsets of Semiconvex sets}

\author[1]{Imre Ruzsa}
\author[2]{Jozsef Solymosi}

\affil[1]{Alfr\'ed R\'enyi Institute of Mathematics, 
Budapest, Hungary}
\affil[2]{Department of Mathematics, 
University of British Columbia, 
Vancouver, BC, Canada}
\date{}                     
\begin{document}

\maketitle

\abstract{We investigate additive properties of sets $A,$ where $A=\{a_1,a_2,\ldots ,a_k\}$ is a monotone
increasing set of real numbers, and the differences of  consecutive
elements are all distinct. It is known  that $|A+B|\geq c|A||B|^{1/2}$ for any
finite set of numbers $B.$ The bound is tight up to the constant
multiplier. We give a new proof to this result using bounds on crossing numbers of geometric graphs. 
We construct examples showing the limits of possible improvements. In particular, we show that there are arbitrarily large
sets with different consecutive differences and sub-quadratic sumset size.}

\section{Introduction}
\label{sec:1}
Given two sets of numbers, $A$ and $B$, the {\em sumset}
of $A$ and $B$, denoted by $A+B,$ is
\[
A+B = \{a+b: a,b \in A \text{ and } b \in B\}
\]
Let $A=\{a_1,a_2,\ldots ,a_k\}$ be a finite set of real numbers
with the property that
\begin{equation}
a_i-a_{i-1}<a_{i+1}-a_i
\end{equation}
for any $1<i<k.$ Sets with this property are said to be {\em
convex} sets.

Improving on a result of Hegyv\'ari \cite{HE},  Elekes, Nathanson, and
Ruzsa \cite{ENR} proved that if $A$ is convex, then
$|A+B|\geq ck^{3/2}$ for any set $B$ with $|B|=k.$ A set, $A,$ has distinct consecutive differences if for
any $1\leq i,j\leq k,$ $a_{i+1}-a_i=a_{j+1}-a_j$ implies $i=j.$

The following theorem, which was proved in \cite{RSS}, generalizes the result of 
Elekes, Nathanson, and Ruzsa.

\begin{theorem}[\cite{RSS}]\label{main}
Let $A$ and $B$ be finite sets of real numbers with $|A| = k$
and $|B| = \ell.$ If $A$ has distinct consecutive differences,
then
\[
|A+B| \geq c{k\sqrt{\ell}}.
\]
In particular, if $k = \ell,$ then
\[
|A+B| \geq c{k^{3/2}},
\]
where $c>0$ is an absolute constant.
\end{theorem}

Although the above theorem is tight, i.e. there are sets $A$ and $B$ such that the above bound is sharp up to the multiplicative constant (see in \cite{RSS}), several questions remain open.
Below we list six questions. In all of them $A$ denotes a set with distinct consecutive differences. Here, and later in the paper, we are using the asymptotic notation $f(x)\gg g(x)$ when there is a constant $c$ such that $f(x)\geq g(x)/\log^cf(x)$ holds as $x$ goes to infinity.
\begin{enumerate}
    \item What can we say about the structure of $A$ when there is a set $B,$ $|B|=|A|=k$ so that $|A+B| = O({k^{3/2}}) $?
    \item There is no restriction on $B$ in Theorem \ref{main}.  What happens if $B$ also has distinct consecutive differences?
    \item What is the minimum size of $A+A$?
    \item What is the best bound on $|A+B|$ if $A$ is convex?
    \item What is the best bound on $|A+A|$ if $A$ is convex?
       \item What is the minimum size of $A+B$ if $A$ and $B$ are convex?
    
\end{enumerate}

All questions above are open. In this paper we will consider the first three questions. 
The following list summarises the best known estimates. 

\begin{enumerate}
    \item If there is a set $B,$ $|B|=|A|=k$ so that $|A+B| = O({k^{3/2}}) $ then the sumset is evenly distributed, $E_{1.5}(A,B)\ll k^{9/4}.$ (in this paper)
    \item Even if both $A$ and $B$ have distinct consecutive differences, there are examples when $|B|=|A|=k$ and $|A+B| = c{k^{3/2}}. $ (in this paper)
    \item There are constructions for sets $A$ such that $|A+A|\leq |A|^{2-c}$ with some $c>0.1$ (in this paper) No better lower bound is known than what follows from Theorem \ref{main}.   
    \item What is the best bound on $|A+B|$ if $A$ is convex? No better lower bound is known than what follows from Theorem \ref{main}, and no construction is known showing $|A+B|\leq (|A||B|)^{1-c}$ with some $c>0.$
    \item If $A$ is convex then $|A+A|\gg |A|^{30/19}$ according to a recent result of Rudnev and Stevens in \cite{RS}, building on earlier results in \cite{LiRo,SchSh,Sh}.
    \item If $A$ and $B$ are convex and $|A|=|B|$ then $|A+B|\gg |A|^{30/19}$  (Ilya Shkredov \cite{ISh}). No construction is known showing $|A+B|\leq (|A||B|)^{1-c}$ with some $c>0.$
\end{enumerate}

There are interesting works relaxing and strengthening the notation of convex sequences, like in \cite{Sch}, \cite{HNR}, and \cite{SW}.

\medskip

\section{Lower bounds using crossing numbers}

\subsection{Proof of Theorem \ref{main}}
In this section we offer a new proof for Theorem \ref{main} by giving a bound in terms of additive energy. As the original proof in \cite{RSS}, this is also a simple proof but here we are using graph theory, the crossing number bound, which gives more information about the structure of $A.$ The variants of the crossing bounds we are going to use are all originate from the classical crossing bound by 
Ajtai, Chv\'atal, Newborn, and Szemer\'edi in \cite{ACNSz}.

\medskip
\begin{proof} ({\em of Theorem \ref{main}.}) For the given sets $A=\{a_1,a_2,\ldots ,a_k\}$ in which the consecutive differences are all distinct and an arbitrary set $B=\{b_1,\ldots,b_\ell\}$ we define a geometric graph, $G$. The vertices of the graph are points on the $x$ axis, the values of the sumset $A+B.$ Two vertices, $u,v\in \{A+B\}$ are connected by an edge, an upper semicircular arc,  iff there is an $a_i\in A$ and $b_j\in B$ such that $u=a_i+b_j$ and $v=a_{i+1}+b_j.$ So $G$ consists of translates of the path, $P,$ with vertex set $A$ where the consecutive vertices are connected by a semicircular arc (Fig.1). Since the consecutive differences are all different in $A,$ the graph has no multiple edges. We are going to bound the crossing number of this graph from below and from above to get a bound on $|A+B|.$ For the upper bound, note that any two translates of $P$ have at most $2k-1$ crossings, so the number of crossings in $G$ is at most 
\begin{equation}\label{lower_b}
cr(G)\leq \binom{|B|}{2}(2k-1)\leq |B|^2{k}.
\end{equation}

\begin{figure}[ht]\label{graph}
\centering
\includegraphics[scale=.4]{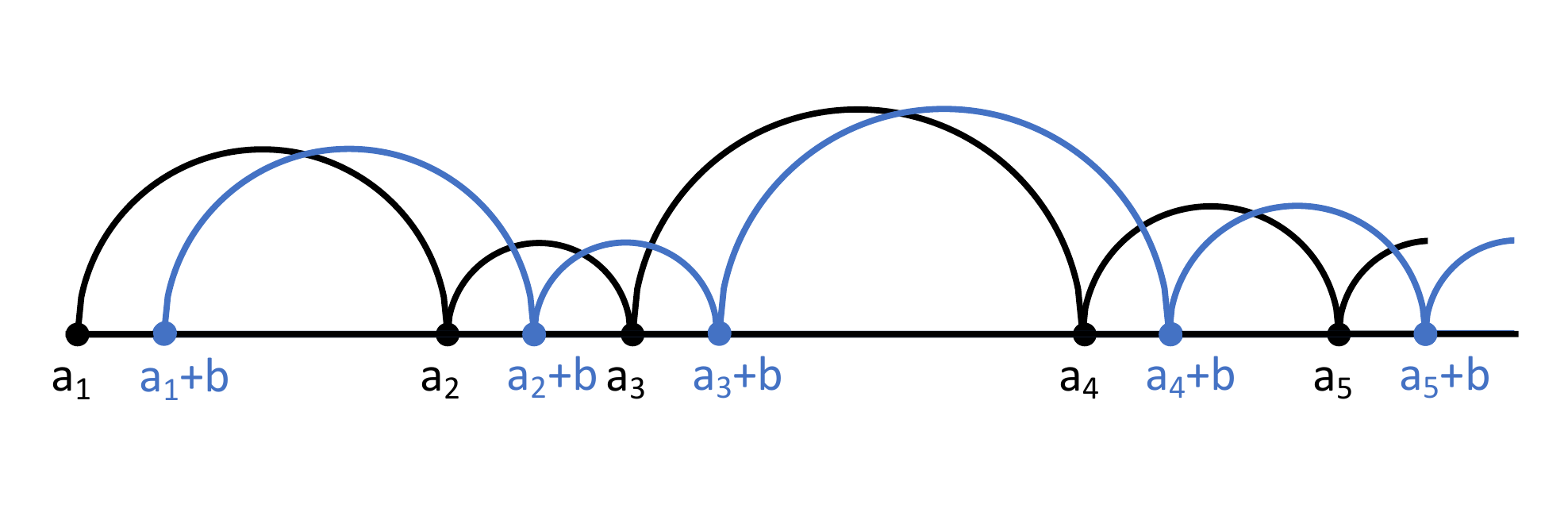}
\caption{The number of crossings between two translates of $A$ is at most $2k-1.$}
\end{figure}

There are various bounds on the crossing number of graphs. In our case this is the convex crossing number (which is the same as the 1-page crossing number) applies. The lower bound is a bit better than for general crossing numbers. It was proved in \cite{SSSZV1} that for $|B|(k-1)$ edges and $|A+B|$ vertices the number of crossings is at least
\begin{equation}\label{cross}
cr(G)\geq \frac{(|B|(k-1))^3}{27|A+B|^2},
\end{equation}

which proves Theorem \ref{main}, with the constant $c=1/\sqrt{27}\approx 0.19,$  since 
\begin{equation}\label{crossing}
  |B|^2{k} \geq \frac{(|B|(k-1))^3}{27|A+B|^2},
\end{equation}
implies that 

$$|A+B| \geq \frac{|A||B|^{1/2}}{\sqrt{27}}.$$

\qed
\end{proof}

\medskip
Our goal is not only to give another simple proof to this result, but to understand the structure of $A$ a bit better when the above bound is close to being tight. One would expect that uneven degree distributions might improve the bound. If in our graph, $G,$ a vertex $v\in \{A+B\}$ has degree $d$ that means that there are at least $d/2$ ways to write $v$ as $v=a+b,$ where $a\in A$ and $b\in B.$ Using the usual notation in additive combinatorics, $r_{A+B}(x)$ denotes the number of ways to write $x$ as $x=a+b,$ where $a\in A$ and $b\in B.$ The additive energy of the two sets is defined as 

$$E(A,B)=E_2(A,B)=\sum_{x\in\{A+B\}}r_{A+B}^{2}(x).$$

In a similar way, for any $\alpha>1$ one can define 

\[
E_\alpha(A,B)=\sum_{x\in\{A+B\}}r_{A+B}^{\alpha}(x).
\]

We are going to use a crossing number bound from \cite{PST} to get an estimate on $E_{1.5}(A,B)$ when $A$ has distinct consecutive differences. 

\begin{theorem}[Theorem 1 in \cite{PST}]
For any simple graph $G$ on $n$ vertices with vertex degrees $d_1\geq d_2\geq \ldots \geq d_n$ we have
\[
cr(G)\geq \frac{1}{36000n}\sum_{i=1}^nid_i^3-4.01n^2.
\]
\end{theorem}

Note that for any sequence of positive real numbers $d_1\geq d_2\geq \ldots \geq d_n$ we have the inequality

\[
\sum_{i=1}^nid_i^3\leq \left(\sum_{i=1}^nd_i^{3/2}\right)^2\leq (\ln{n}+1)\sum_{i=1}^nid_i^3,
\]

which leads us to 

\[
|B|^2k+4.01|A+B|^2\geq \frac{c_1}{|A+B|\ln{|A+B|}}\left(\sum_{x\in\{A+B\}}r_{A+B}^{3/2}(x)\right)^2.
\]
We use this bound in the case when $|B|=|A|=k.$   By Theorem \ref{main} and the inequality above
we get the following result.

\begin{theorem}
If $A$ is a set where the consecutive differences are distinct, and $B$ is an arbitrary set such that $|A|=|B|,$ then \footnote{We remind the reader that the $\gg$ notation hides a possible $\log^c$ multiplier as we defined it in the Introduction.}
\[
|A+B|\gg \left(\sum_{x\in\{A+B\}}r_{A+B}^{3/2}(x)\right)^{2/3}=E_{1.5}(A,B)^{2/3}.
\]
\end{theorem}
This bound - up to the log factor - is an improved version of Theorem \ref{main}. By Jensen's inequality the right hand side is minimal if $r_{A+B}(x)$ is the average, $k^2/|A+B|,$ for all $x\in\{A+B\},$ and then $|A+B|\approx k^{3/2}.$

\medskip

There is another, more direct way to measure the smoothness of the degree sequence, to count the number of $x\in \{A+B\}$ which have high multiplicity, i.e. there are many ways to express $x$ as a sum. For this, let us choose a set $S\subset \{A+B\}$ which is a collection of such elements.

\begin{theorem}\label{delta}
Suppose that $A$ is a set with consecutive differences being all distinct, and $B$ is an arbitrary set. Let  $S\subset \{A+B\}$ a set such that 
\begin{equation}\label{core}
\sum_{x\in S}r_{A+B}(x)\geq \frac{|A||B|}{\Delta}.
\end{equation}
Then there is a constant, $c_\Delta>1/(2\Delta)^3,$ so the following inequality holds
\begin{equation}\label{delta_2}
|A+B|\geq \frac{c_\Delta|B||A|^2}{|S|}.
\end{equation}
\end{theorem}

As in the previous theorem, this result implies Theorem \ref{main}. Setting $S=\{A+B\}$ gives $|A+B|=\Omega\left(|B|^{1/2}|A|\right),$ as in Theorem \ref{main}.
At the other end, if $|S|\approx |A|,$ then the sumset is as large as possible, it is $\Omega(|A||B|).$

\begin{proof}
We will work with $G$ as defined in the proof of Theorem \ref{main} and we are going to use a simple crossing bound from \cite{PST}.
\begin{lemma}[Lemma 2.1 in \cite{PST}]
Let $G(U,V)$ be a bipartite graph with vertex classes $U$ and $V,$ and
suppose that its number of edges satisfies $e \geq 6 \max(|U|,|V|).$ Then we have
\[
cr(G(U,V))\geq \frac{e^3}{108|U||V|}.
\]
\end{lemma}\label{bicross}
From the proof of Theorem \ref{main} we know that $cr(G)\leq |A||B|^2.$  There are two simple cases to consider; 
when a $\delta$-fraction of the edges with an endpoint in $S$ are inside of $S,$ 
or the $(1-\delta)$-fraction of them connect $S$ with the outside of $S.$ (We are going to optimize for $0<\delta<1$ at the end of our calculation)
For the first case we apply a classical crossing bound for the induced subgraph of $G$ on $S,$ denoted by $H_S.$ By inequality (\ref{cross}) we have
\[
|A||B|^2\geq cr(H_S)\geq \frac{(\delta|A||B|/\Delta)^3}{27|S|^2}\geq \frac{(\delta|A||B|/\Delta)^3}{27|A+B||S|},
\]
which gives the required inequality
\[
|A+B|\geq \left(\frac{\delta}{3\Delta}\right)^3\frac{|B||A|^2}{|S|}.
\]
In the second case we can assume that $|A||B|^2\geq 6|A+B|,$ so Lemma \ref{bicross} is applicable with $V=S$ and $U=\{A+B\}\setminus S.$
\[
|A||B|^2\geq \frac{e^3}{108|U||V|}> \frac{((1-\delta)|A||B|/\Delta)^3}{125|S||A+B|},
\]
which gives 
\[
|A+B|> \left(\frac{1-\delta}{5\Delta}\right)^3\frac{|B||A|^2}{|S|}.
\]

Choosing $\delta$ to make the constant multipliers in the two cases equal, we see that one can set $c_\Delta>1/(2\Delta)^3$ in Theorem \ref{delta}. 

\qed
\end{proof}

\medskip

One can be more specific in defining $S,$ we can bound the number of $x\in A+B$ such that $r_{A+B}(x)\geq t.$

\begin{corollary}
Suppose that $A$ is a set with consecutive differences being all distinct, and $B$ is an arbitrary set. Let $t>1$ an arbitrary integer and $S_t\subset \{A+B\}$ a set such that 
\begin{equation}
\min_{x\in S_t}\{r_{A+B}(x)\}\geq t.
\end{equation}
Then 

\[
|S_t|<\frac{3|A+B|^{1/2}|A|^{1/2}|B|}{t^{3/2}}.
\]

\end{corollary}

\begin{proof}
Set 

\[
\Delta=\frac{|A||B|}{t|S_t|}
\]
so the following inequality holds
\[
\sum_{x\in S_t}r_{A+B}(x)\geq|S_t|t=\frac{|A||B|}{\Delta},
\]
and we can apply Theorem \ref{delta}. Solving inequality (\ref{delta_2}) for $|S_t|$ we get the desired result. 

\qed
\end{proof}

\medskip
Another advantage of using crossing numbers in this context is that there are effective bounds for multigraphs, so we get the following result.

\begin{claim}
Let $A$ and $B$ be finite sets of real numbers with $|A| = k$
and $|B| = \ell.$ If in $A$ any consecutive difference, $a_i-a_{i-1},$ has multiplicity at most $m$
then
\[
|A+B| \geq \frac{ck\sqrt{\ell}}{\sqrt{m}},
\]
where $c>0$ is a universal constant.
\end{claim}

\begin{proof}
The claim follows directly from the crossing bound for multigraphs (see e.g. in \cite{Szek}) and from our upper bound in equation (\ref{lower_b}).
\[
k\ell^2\geq cr(G)\geq \frac{c(k\ell)^3}{m|A+B|^2}.
\]

\qed
\end{proof}

\subsection{A small improvement on the crossing number bound}
In the above results we used crossing number bounds, however we could have used a slightly better graph parameter for our purposes. It might be interesting if one would like to improve the 
constant multiplier in the lower bounds. For any graph, $G_n,$ an interval drawing is given by a one-to-one map, $\Phi : V(G_n)\rightarrow \mathbb{R},$ where an edge $(v_i,v_j)\in E(V(G_n)$ maps to the interval $[\Phi(v_i),\Phi(v_j)].$ Two vertex-disjoint edges, $(v_j,v_j)$ and $(v_k,v_\ell),$ are intersecting (under the map $\Phi$) if they share an interior point,
$$[\Phi(v_i),\Phi(v_j)]\cap [\Phi(v_k),\Phi(v_\ell)]\neq \emptyset. $$ 

\begin{definition}
For a given graph $G_n,$ the intersecting number, $int(G_n),$ is the lowest number of edge intersections of an interval drawing of $G_n.$
\end{definition}

Our upper bound in (\ref{lower_b}) holds for the intersection number of $G_n$ and it is clear that 
$$int(G_n)\geq cr(G_n).$$
The inequality is strict if one edge (interval) contains another inside. The maximum number of  crossing-free edges in $G_n$ is about $2n,$ while for intersection-free $G_n$ it is at most $1.5n.$ Using the probabilistic approach of Sz\'ekely (as we will see in the proof of Claim \ref{const}), one can show that if the number of edges, $e,$ is at least $2.25n,$ then 
\[
int(G_n)\geq \frac{0.0658e^3}{n^2},
\]
which is better than the inequality used in equation (\ref{crossing}) since $1/27\approx 0.037 < 0.0658$. As an example we show the following bound:
\begin{claim}\label{const}
Let $A=\{a_1,a_2,\ldots ,a_k\}$ be a finite set of real numbers
with distinct consecutive differences which are not too far from each other, i.e.
\begin{equation}\label{double}
a_i-a_{i-1}\leq 2(a_{j}-a_{j-1})
\end{equation}
for any $1<i,j\leq k,$ and let $B\subset \mathbb{R}$  an arbitrary finite set. Then 
\[
|A+B|\geq \frac{2}{3\sqrt{3}}|A||B|^{1/2}.
\]
\end{claim}

\begin{proof}
To see this, note that by condition (\ref{double}) if a subgraph of $G_{|A+B|}$ on $n$ vertices is intersection free, then it has at most $n-1$ edges. 
If the number of edges is at least $1.5n$ then there are at least $0.5n$ intersections. The number of edges in $G_{|A+B|}$ is $(k-1)|B|\approx |A||B|,$ 
so let's select a random subgraph of $G_{|A+B|}$ choosing the vertices independently at random with probability 
$p=\frac{1.5|A+B|}{|A||B|}.$ Then the expected number of vertices is $\frac{1.5|A+B|^2}{|A||B|},$ the number of edges is $\frac{(1.5|A+B|)^2}{|A||B|},$ and
the expected number of intersections is at least $p^4\cdot int(G_{|A+B|}).$ From here we have

\[
\left(\frac{1.5|A+B|}{|A||B|}\right)^4\cdot int(G_{|A+B|})\geq 0.5\cdot \frac{1.5|A+B|^2}{|A||B|},
\]
and 
\[
|A||B|^2\geq int(G_{|A+B|})\geq 0.5\cdot \frac{(|A||B|)^3}{1.5^3|A+B|^2}.
\]
\qed
\end{proof}

Note that in the proof we only used the weaker condition that $a_{i+2}-a_i\neq a_j$ for any $1\leq i,j\leq k.$
\section{Constructions}

\subsection{Both $A$ and $B$ have distinct consecutive differences.}

In this section we are going to show constructions which indicate the limitations of possible lower bounds on the size of sumsets forced by 
local conditions, like distinct consecutive differences in a set.

\begin{theorem}
For arbitrary large integer $k,$ there are are sets, $A$ and $B,$ such that in both sets the consecutive differences are distinct, $|A|\approx |B|\geq k$ and $|A+B|\leq c|A||B|^{1/2},$ where $c>0$ is a universal constant.
\end{theorem}

\begin{proof}
Let $n$ and $k$ be relative prime numbers, $1 < n < k,$ such that $k-n$ is small. (we will assign their exact values later) Set $A$ is defined as

$$A=\{jk\ |\ 0 \leq j < n/2\}.$$ 

Between two consecutive numbers of $A$ there is at least one multiple of  $n.$ Let us choose one of them and add it to $A.$ In this way we have  $n-1$ numbers, all below $kn/2.$
All consecutive differences are distinct since the difference is either of the form $jk-in$ or
$in-jk.$  If $jk-in = j'k-i'n,$ then $(j-j')k = (i -i')n,$ so $n \mid j- j',$
$j = j'$ and then $i = i'.$
The same holds for the  $in- jk = i'n -j'k$ case.
If $jk - in = i'n - j'k,$ then $(j + j')k = (i + i')n,$ so $n \mid j + j', j + j' < n,$
so $j = j' = 0,$ which is not possible.

The construction of $B$ follows the same algorithm, using numbers $m, r.$

$$B=\{jr\ |\ 0 \leq j < m/2\}.$$

We choose $m, r$ in a way that $A + B$ is small.
To achieve this let $1 < a < b < c < d$ be four pairwise relative prime numbers close to each other.
e.g. $a = 6t+1, b = 6t+2, c = 6t+3, d = 6t+5.$ Let $n = ab, k = cd, m = ac, r = bd.$
The elements in both sets are less than  $abcd/2,$ so the sumset is subset of $[0, abcd).$  All elements of the sumset are divisible by one of the numbers $ a, b, c, d,$  so its cardinality is less than $4bcd.$ The cardinality of the sets is at least $a^2$ and 
\[
4bcd<(4+\varepsilon)(a^2)^{3/2}.
\]
\qed
\end{proof}

\subsection{Bounding the size of $A+A.$}

The best known lower bound on $|A+A|,$ where $A$ has distinct consecutive differences follows from Theorem \ref{main}. One can use the structure of the graph $G$ defined by $A+A$ to improve the constant multiplier in Theorem \ref{main}, but it is still $|A+A|=\Omega(|A|^{3/2}).$ In this direction the best result is due to Schoen \cite{Sch} who proved a better bound if $A$ is a $tdcd$ set. Schoen calls a a set a $tdcd$-set (totally distinct consecutive differences) if for every fixed $1\leq d < |A|,$ all differences
$a_i-a_{i-d},$ where $d < i < n,$ are distinct. For $tdcd$ sets Schoen proved that there is a constant $c>0$ such that $|A+A|\geq c|A|^{3/2+c}.$

The next construction shows that there are sets $A$ with distinct consecutive differences such that $|A+A|\leq c|A|^{2-c}.$ 

\begin{theorem}
There is a constant, $c>0.1,$ such that for arbitrarily large $n$ there is a set $A$ with distinct consecutive differences, $|A|\geq n,$ such that $|A+A|=O\left(|A|^{2-c}\right).$
\end{theorem}

\begin{proof}

For the construction of $A$ we are going to use a set of integers with larger difference set than sumset. Searching for sets with many more differences than sums we selected  the set  $S=\{0,1,3,7,12,22,30\}.$
$S$ has 43 distinct pairwise differences and 28 sums. 

If the reader would like to follow the construction with a smaller set then one can perform all steps using the $S=\{0,1,3\}$ set, and then $c>0$ is a bit below 0.1 ($ c\approx 0.92 $), but the steps are easier to check. The selected seven element set is a result of a simple optimization to maximize $c.$ We are going to revisit the selection of $S$ at the end of the proof.

First we construct a set of $k$-dimensional vectors, $Q_k,$ in a sequence that consecutive vectors have distinct (vector) differences.
Let us consider $S$ as the vertex set of a complete digraph on 7 vertices, and assign a value to every edge as follows: If the edge is $v_i\rightarrow v_j$ then its value is defined as $w(i,j)=v_i-v_j.$ For example, $w(2,4)=-6,$ and $w(3,3)=0.$ We will consider
walks in this digraph. The first walk is an Euler tour starting and ending in $v_1.$ 
Listing the indices of the vertices in sequence as the tour goes we have e.g.

$$E=\{1,3,5,2,6,4,7,2,4,1,5,7,3,6,1,2,3,4,5,6,7,1,7,5,3,7,4,6,5,1,6,2,5,4,3,1,4,2,7,6,3,2,1\}.$$

Using the values assigned to the vertices we have our first multiset of (one dimensional) vectors, 

{\tiny $$Q_1=\{0,3,12,1,22,7,30,1,7,0,12,30,3,22,0,1,3,7,12,22,30,0,30,12,3,30,7,22,12,0,22,1,12,7,3,0,7,1,30,22,3,1,0\}.$$} 
By the construction, all consecutive differences are distinct, we used all differences but 0. We define $Q_{k+1}$ recursively, using $Q_k.$ There will be $43|Q_k|=43^{k+1},$ not necessarily distinct $k+1$ dimensional vectors in $Q_{k+1},$ such that all consecutive differences are distinct.
The first $k$ coordinates are periodically repeating vectors, lets take $k$ copies, blocks, of $Q_k,$ one after the other. We choose the last, $(k+1)$-th coordinate for every vector in blocks using $Q_1$ as follows.
In the first block all $(k+1)$-th coordinates are 0. In the second block we alternate 3 and 0, as $3,0,3,0,\ldots,0,3.$ The third is $12,3,12,3,\ldots,3,12.$ For the $i$-th block, $(i>1),$ we use the $i$-th and $i-1$-th entries of $Q_1$ and alternate them starting (and ending) with the $i$-th. Note that the first and last vectors are identical.

In this way all consecutive vectors have distinct differences. If the differences in the last coordinate are the same in two pairs of consecutive vectors, then there are three cases. 
\begin{itemize}
    \item  First, the two pairs are selected from the same $Q_k$ block. In this case the differences of the first $k$ coordinates are distinct by induction. 

\item The second case when one of the pairs is between two blocks. There is only one pair of blocks where the last and the first vectors have a given difference in the last coordinate, so the other pair is inside one block. But in the difference vector of pairs between blocks the first $k$ coordinates are zero, while differences of consecutive vectors in the same block have some nonzero coordinates among the first $k.$ 

\item The third case is when one pair is in one block and the other is in another block. It is only possible if one pair is in a block with last coordinates $a,b,a,\ldots,b,a$ and the other is in the block with $b,a,\ldots,a,b,$ but in this case the positions of the two pairs relative to their blocks are different due to parity, so by induction the differences are distinct in the first $k$ coordinates.
\end{itemize}

\begin{figure}[ht]\label{blocks}
\centering
\includegraphics[scale=.4]{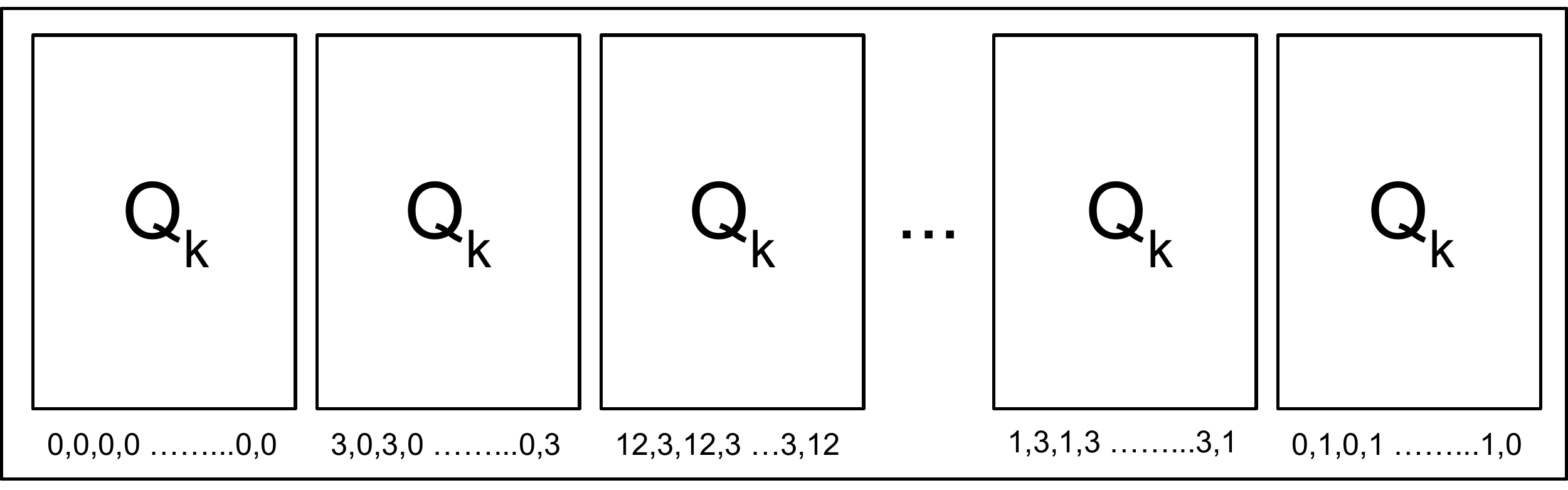}
\caption{The recursion getting $Q_{k+1}$ from $Q_k.$}
\end{figure}

Our next step is to construct an increasing sequence of numbers using the vectors in $Q_k$ keeping the same order and keeping the property that consecutive elements have distinct differences.
If $\Vec{v_i}\in Q_k$ has coordinates $\Vec{v_i}^T=[\nu_1,\nu_2,\ldots,\nu_k],$ then let us define $b_i=\nu_1+100\nu_2+100^2\nu_3+\ldots+100^{k-1}\nu_k.$ With this definition $b_i\leq 100^k,$ and the consecutive differences (in the order of the vectors in $Q_k$) are all distinct. To make the sequence monotone increasing, we define 

\[A=\left\{a_i \text{  }|\text{  }  a_i=b_i+i100^k, 1\leq i\leq 43^k\right\}.\]

The sumset, $A+A,$ consists of sums $a_i+a_j=b_i+b_j+(i+j)100^k.$ From the selection of the initial set, 
$S,$ we see that $\left|\left\{b_i+b_j | 1\leq i,j\leq 43^k\right\}\right|\leq 28^k,$
so we have the bound $|A+A|\leq 28^k\cdot2\cdot43^k$ while $|A|=43^k.$ By this construction we get a set $A,$ with distinct consecutive differences and $|A+A|\leq 2|A|^{2-c}$ where 
\[
c= \frac{\log{\frac{43}{28}}}{\log{43}}\approx 0.11406.
\]

From the construction we see that we needed a set $S$ with small sumset and large difference set. We were searching among Sidon sets, i.e. sets where all pairwise sums are distinct. In such sets, if the size of $S$ is $x$ then $|S+S|=\binom{x}{2}+x$ and $|S-S|=x(x-1)+1.$ As we are looking for a large $c$ above, we want to find the maximum of the function

\[
\frac{\log{\frac{x(x-1)+1}{\binom{x}{2}+x}}}{\log{x(x-1)+1}},
\]

for positive $x.$ The maximum is $0.114058\ldots$ when $x \approx 6.99618,$  so we selected $S$ to be a 7-element Sidon set, $\{0,1,3,7,12,22,30\}.$

\qed

\end{proof}

\section{Acknowledgements}
We would like to thank Endre Szemer\'edi and Ilya Shkredov for helpful discussions. We appreciate useful comments from the anonymous referee.
Research was supported in part by the  OTKA K 119528 and K129335 grants. The second author was supported by NSERC Discovery and NKFI KKP 133819 grants.

\end{document}